\documentclass{amsart}      
\usepackage{amssymb,amsmath,graphicx,enumerate,xcolor,url,mathrsfs,enumitem,hyperref}
\usepackage[all]{xy}
\newtoks\prt 
 
\newtheorem{thm}{Theorem}[section]
\newtheorem{ques}[thm]{Question} 
\newtheorem{lemma}[thm]{Lemma} 
\newtheorem{prop}[thm]{Proposition} 
\newtheorem{cor}[thm]{Corollary}

\theoremstyle{definition} 
 
\newtheorem{remark}[thm]{Remark}

\def\eqn#1$$#2$${\begin{equation}\label#1#2\end{equation}}

\def\F{\mathcal F}

\def\diam{\operatorname{diam}}

\def\en{\mathbb N} 
\def\er{\mathbb R} 
\def\zet{\mathbb Z}

\def \ext {\operatorname{ext}}

\def\span{\operatorname{span}} 
\def\sgn{\operatorname{sgn}}

\def \reg {\partial _{\kern1pt\text{reg}}} 
 
\def\iff{\Longleftrightarrow} 

\def\x{\boldsymbol{x}}

\def\Lip{\operatorname{Lip}}

\def\ip#1#2{\left\langle#1,#2\right\rangle}

\def\wde#1{\widetilde{\delta}\left(#1\right)}
\def\de#1{\delta\left(#1\right)}

\newcommand{\norm}[1]{\left\|#1\right\|}

\newcommand{\ca}[1]{\operatorname{ca}\left(#1\right)}
\newcommand{\wca}[2][]{\widetilde{\operatorname{ca}}_{#1}\left(#2\right)}

\newcommand{\abs}[1]{\left|#1\right|}
\newcommand{\setsep}{;\,}

\newcommand\restr[2]{{
  \left.\kern-\nulldelimiterspace 
  #1 
  \littletaller 
  \right|_{#2} 
  }}

\newcommand\absb[2]{\csname#1l\endcsname|#2\csname#1r\endcsname|}
\newcommand\normb[2]{\csname#1l\endcsname\|#2\csname#1r\endcsname\|}

\begin{document}

\title[Lipschitz-free spaces over unifromly discrete are 3-Schur]{Lipschitz-free spaces over uniformly discrete metric spaces are $3$-Schur}
\author[{M. C\' uth}]{Marek C\'uth}
\author[{O. Kalenda}]{Ond\v{r}ej F.K. Kalenda}

\email{cuth@karlin.mff.cuni.cz}
\email{kalenda@karlin.mff.cuni.cz}

\address{Department of Mathematical Analysis \\
Faculty of Mathematics and Physic\\ Charles University\\
Sokolovsk\'{a} 83, 186 \ 75\\Praha 8, Czech Republic}
\email{cuth@karlin.mff.cuni.cz}
\email{kalenda@karlin.mff.cuni.cz}

\subjclass[2010]{46B04; 51F30; 54E50 (Primary); 46B80; 46B20 (Secondary)}
\keywords{Lipschitz function; Lipschitz-free space; uniformly discrete space;  quantitative Schur property, 1-strong Schur property}


\begin{abstract} We prove that the Lipschitz-free space over any uniformly discrete metric space has the $3$-Schur property
\end{abstract}

\maketitle

\section{Introduction}

The study of Lipschitz-free Banach spaces has become an active area of research over the past two decades. These spaces arise in a variety of contexts and are known under several names, including Arens-Eells spaces and transportation cost spaces. One of their key features is that they provide a canonical way to associate a Banach space $\F(M)$ to a metric space $M$, allowing metric properties of $M$ to be translated into functional-analytic properties of $\F(M)$, and vice versa.

A notable result in this direction is a recent theorem of Aliaga et al.~\cite{p1u}, which states that the Lipschitz-free space $\F(M)$ over a complete metric space $M$ has the \emph{Schur property} (i.e., weak and norm convergences of sequences coincide) if and only if $M$ is \emph{purely $1$-unrectifiable} (that is, $M$ contains no bi-Lipschitz copy of a compact subset of the real line with positive Lebesgue measure, see \cite[Corollary 1.12]{p1u}). For further results relating properties of $\F(M)$ to those of the underlying metric space $M$, we refer the reader, for instance, to \cite{GPZ18, PZ18}.

Lipschitz-free spaces also provide examples of Banach spaces exhibiting previously unknown phenomena. For instance, in \cite[Example 8.3]{qschur-rim25}, the second author constructed a Lipschitz-free space that was the first known Banach space with the $1$-strong Schur property but not the $1$-Schur property. For additional constructions of Lipschitz-free spaces with novel properties, see, e.g., \cite{KV23}.

In the present paper, we provide a quantitative strengthening of the aforementioned result of \cite{p1u} for certain natural subclasses of purely $1$-unrectifiable metric spaces. To achieve this, we employ a quantitative version of the Schur property introduced in \cite{qschur}. This approach is part of the broader program of quantitative functional analysis, which, roughly speaking, aims to replace qualitative implications by more precise inequalities between appropriately defined quantities.

Let us recall basic quantities related to the Schur property.
Given a bounded sequence $(x_n)$ in a Banach space, we set
$$\ca{x_n}=\inf_{n\in\en} \diam\{x_k\setsep k\ge n\},$$
i.e., $\ca{x_n}$ is the oscillation of sequence $(x_n)$. Note that $\ca{x_n}=0$ if and only if $(x_n)$ is Cauchy
(hence convergent). Hence $\ca{x_n}$ measures how far is the sequence from being Cauchy.
For a bounded sequence $(x_n)$ we further set
$$\de{x_n}=\sup\{\ca{x^*(x_n)}\setsep x^*\in X^*,\norm{x^*}\le 1\}.$$
Then $\de{x_n}=0$ if and only if $(x_n)$ is weakly Cauchy, so this quantity measures how far the sequence is from being weakly Cauchy. Following \cite{qschur} we say that a Banach space $X$ is $c$-Schur for some $c\ge1$ if
$$\ca{x_n}\le c\de{x_n}\mbox{ for each bounded sequence }(x_n)\subset X.$$
In \cite{qschur} it is proved that the classical space $\ell_1$ is $1$-Schur and that there are some Schur spaces which are not quantitatively Schur. Moreover, in \cite{qschur-dp} it is proved that $X^*$ is $1$-Schur whenever $X$ is isometric to a subspace of $c_0(\Gamma)$ for a set $\Gamma$. This result was recently used to show (for a complete metric space $M$) that $\F(M)$ is $1$-Schur provided $M$ is purely $1$-unrectifiable and \emph{proper} (i.e., closed bounded sets are compact). This follows by combining \cite[Proposition 17]{petitjean-schur} and \cite[Theorem 3.2]{p1u}. We note that there is also another way to quantify the Schur property, inspired by the notion of $1$-strong Schur property from \cite{Go-Ka-Li}. For details and a precise relationship of these two ways of quantification, see
 Section~\ref{sec:qschur} below.
 
Recently, in \cite{qschur-rim25} the second author investigated the quantitative Schur property in Lipschitz-free spaces over uniformly discrete spaces. Note that a uniformly discrete metric space is obviously purely $1$-unrectifiable, so its Lipschitz-free space is Schur by \cite{p1u}. However, such a metric space is proper if and only if all bounded subsets are finite, which is very restrictive.
On the other hand, the structure of uniformly discrete spaces is quite simple, so it is natural to ask whether their Lipschitz-free spaces are automatically quantitatively Schur. In \cite{qschur-rim25} some counterexamples were given witnessing that the constant $1$ is in general unreachable. More precisely, there is a uniformly discrete metric space $M$ (made from a graph using the shortest-path metric) which is $3$-Schur and not $c$-Schur for any $c<3$. In the present paper we prove that the constant $3$ is optimal. More precisely, the main result is the following:

\begin{thm}\label{T:main}
    Let $M$ be a uniformly discrete metric space. Then $\F(M)$ is $3$-Schur. Moreover, the constant $3$ is optimal.
\end{thm}

Let us note that our Theorem~\ref{T:main} contributes to the recent line of research focused on Banach space theoretical properties of Lipschitz-free spaces over uniformly discrete spaces (hereafter referred to as LFUD spaces). This line of research was initiated in \cite[Proposition 4.4]{kalton-collect04}, where N. Kalton observed that every LFUD space is Schur, it enjoys the Radon-Nikodym property, and the approximation property. One of the major open problems, originally posed by N. Kalton in \cite[p.~185]{kalton-collect04}, asks whether every LFUD space has the bounded approximation property (both positive and negative answer having interesting consequences, see \cite[Problem 18]{GLZ14} for further discussion). It is therefore natural to investigate more deeply the structural properties of LFUD spaces. Several recent results contribute to this direction; for instance, \cite[Corollary 3.7]{APS24} shows that every element of an LFUD space is a convex integral of molecules. Moreover, recent works such as \cite{AM26} and \cite{HM23} study Lipschitz-free spaces over nets in Banach spaces, which form an important subclass of LFUD spaces.

Theorem~\ref{T:main} will be proved in Section~\ref{sec:proof} below with help of several lemmata. The following problem remains open.

\begin{ques}
  Assume $M$ is a complete purely $1$-unrectifiable metric space. Is $\F(M)$ $3$-Schur?  
\end{ques}

We recall that $\F(M)$ is even $1$-Schur if $M$ is additionally proper. Our main results provides a positive answer for uniformly discrete spaces. But there are many complete purely $1$-unrectifiable spaces which are neither uniformly discrete nor proper. In particular, the following special cases remain open.

\begin{ques}
\begin{enumerate}[label=(\roman*)]
    \item Assume that $M$ is a complete discrete metric space (i.e., a complete space without isolated points,
    not necessarily uniformly discrete). Is $\F(M)$ $3$-Schur?
    \item Let $X$ be an infinite-dimensional Banach space (for example $X=\ell_1$, $X=c_0$ or $X=\ell_2$). Let $p\in (0,1)$. Define a new metric on $X$ by $d(x,y)=\norm{x-y}^p$, $x,y\in X$. Is $\F(X,d)$ $3$-Schur?
\end{enumerate}
\end{ques}

We complete our main result by a further special case.

\begin{thm}\label{t:0hyp}
    Let $M$ be a complete $0$-hyperbolic purely $1$-unrectifiable metric space. Then $\F(M)$ is $1$-Schur.
\end{thm}

We recall that a metric space $(M,d)$ is \emph{$0$-hyperbolic} if
$$\forall x,y,z,u\in M\colon d(x,y)+d(z,u)\le \max\{d(x,z)+d(y,u),d(x,u)+d(y,z)\}.$$
The proof is made by combining two known results and is given below in Section~\ref{sec:0hyp}.

A special subclass of $0$-hyperbolic spaces is formed by \emph{ultrametric} spaces. Recall that a metric space $(M,d)$ is ultrametric if it satisfies the following stronger version of triangle inequality:
$$\forall x,y,z\in M\colon d(x,z)\le \max\{d(x,y),d(y,z)\}.$$
It is easy to see that any ultrametric space is $0$-hyperbolic and it is essentially known that ultrametric spaces are purely $1$-unrectifiable (see Section~\ref{sec:0hyp} for more details), so we have the following corollary.

\begin{cor}\label{cor:ultrametric}
Let $M$ be an  ultrametric space. Then $\F(M)$ is $1$-Schur.    
\end{cor}

Let us note that one could also investigate quantitative estimates of other properties of Lipschitz-free spaces. For instance, a quantitative version of weak sequential compactness (WSC), known as $C$-WSC for $C>0$, has been introduced; see, e.g., \cite{wesecom, qschur} for its definition and basic properties. Since it is known that many Lipschitz-free spaces enjoy the WSC property (see, for example, \cite[Corollary 2.5]{ANPP2021}), it would be natural to ask which Lipschitz-free spaces are $C$-WSC for some $C \geq \tfrac{1}{2}$, and to determine the optimal constant whenever possible.\footnote{Note that Lipschitz-free spaces cannot be $C$-WSC for $C < \tfrac{1}{2}$, as this would imply reflexivity by \cite[Proposition 1.2]{qschur}. This contradicts the fact that $\ell_1$ embeds into every infinite-dimensional Lipschitz-free space, see \cite{cuth-doucha-w}.} This question appears to be of interest even in the case of uniformly discrete metric spaces. Indeed, by Theorem~\ref{T:main} and \cite[Proposition 1.1]{qschur}, Lipschitz-free spaces over uniformly discrete spaces are $3$-WSC, but it is not clear whether this constant is optimal. On the other hand, the optimal constant is known for Lipschitz-free spaces over proper purely 1-unrectifiable spaces: these spaces are $L$-embedded (see \cite[Section 6]{APS25}) and hence are $\tfrac{1}{2}$-WSC by \cite[Theorem 1]{wesecom}.

We close this introduction by recalling the basic definitions and results on Lipschitz-free spaces used below.

\subsection*{Preliminaries on Lipschitz-free spaces:} Given a metric space $(M,d)$ with a distinguished point $0\in M$, symbol $\Lip_0(M)$ denotes the spaces of all real-valued Lipschitz functions on $M$ vanishing at $0$ equipped with the norm given by 
\[\|f\|_{\Lip} = L(f) = \sup\Big\{\frac{\abs{f(x)-f(y)}}{d(x,y)}\setsep x\neq y\in M\Big\}.\]
Then $\Lip_0(M)$ is a Banach space. For each $x\in M$, let $\delta(x)$ denote the evaluation functional $f \mapsto f(x)$ on $\Lip_0(M)$. The mapping $x \mapsto \delta(x)$ defines an isometric embedding of $(M,d)$ into $\Lip_0(M)^*$. The Lipschitz-free space over $M$ is then defined by
\[
\F(M):=\overline{\span}\{\delta(x)\setsep x\in M\}\subset \Lip_0(M)^*.
\]
Now we recall some basic facts about Lipschitz-free spaces; for their proofs we refer the interested reader, for instance, to \cite[Section 2]{cuth-doucha-w}. The Lipschitz-free space is independent of the choice of the distinguished point $0\in M$; that is, Lipschitz-free spaces constructed with different distinguished points are linearly isometric.
Given $0\in N\subset M$, the natural mapping $\F(N)\ni\delta(x)\mapsto \delta(x)\in\F(M)$, $x\in N$ extends to an isometric embedding, so $\F(N)$ can be identified isometrically with 
\[
\overline{\operatorname{span}}\{\delta(x)\setsep x\in N\}\subset \F(M),
\]
and we will use this identification throughout without further mention. If $\widetilde{M}$ denotes the completion of $M$, then $\F(M)$ and $\F(\widetilde{M})$ are linearly isometric. Hence, we will often assume without loss of generality that $M$ is complete. The dual space $\F(M)^*$ is linearly isometric to $\Lip_0(M)$, with the duality given by
\[
\ip{f}{\delta(x)}:=f(x),\quad f\in\Lip_0(M), x\in M
\]
and extended uniquely from $\delta(M)$ to a continuous linear functional on $\F(M) = \overline{\span}\; \delta(M)$. On bounded subsets of $\Lip_0(M)$, the weak$^*$ topology coincides with the topology of pointwise convergence. Finally, for $C>0$, if $M$ and $M'$ are $C$-Lipschitz isomorphic metric spaces, then the Banach spaces $\F(M)$ and $\F(M')$ are $C$-isomorphic.\footnote{We say $f:M\to M'$ is $C$-Lipschitz isomorphism if $f$ is bijection and $\Lip(f)\cdot \Lip(f^{-1})\leq C$. If, in addition, $M$ and $M'$ are Banach spaces and $f$ is moreover linear, we say $f$ is $C$-isomorphism.}

\section{On quantification of the Schur property}\label{sec:qschur}

As recalled in the introduction, a Banach space $X$ is Schur if any weakly convergent sequence is norm convergent.
Further, $X$ is called $c$-Schur (where $c\ge 1$) if $\ca{x_n}\le c\de{x_n}$ for each bounded sequence $(x_n)\in X$. We will use the following characterization of the $c$-Schur property, which easily follows from \cite[Proposition 5.1$(i_c)\iff(ii_c)$]{qschur-rim25}.

\begin{lemma}\label{l:c-schur-char}
    Let $X$ be a Banach space and let $c\ge 1$. Then $X$ is $c$-Schur if and only if
    $$\limsup_n \norm{x_n}\le c \sup \{\limsup_n \abs{x^*(x_n)}\setsep x^*\in B_{X^*}\}$$
    for any bounded sequence $(x_n)$ in $X$.
\end{lemma}

In the literature also another quantification of the Schur property is used.  It is so-called $1$-strong Schur property introduced in \cite{Go-Ka-Li} and used also in \cite{p1u}. Its $c$-version was introduced in \cite{qschur-rim25} and it was compared there with the $c$-Schur property. So, we say that a Banach space $X$ has the \emph{$c$-strong Schur property} if for any $\delta\in(0,2]$ and any $\varepsilon>0$ any normalized $\delta$-separated sequence in $X$ admits a subsequence $(\frac{2c}{\delta}+\varepsilon)$-equivalent to the standard basis of $\ell_1$.

To formulate its relationship to the $c$-Schur property, we recall two more quantities. Given a bounded sequence $(x_n)$ in a Banach space, we set
$$\begin{aligned}
    \wca{x_n}&=\inf\{ \ca{x_{n_k}}\setsep (n_k)\mbox{ strictly increasing}\},\\
    \wde{x_n}&=\inf\{ \de{x_{n_k}}\setsep (n_k)\mbox{ strictly increasing}\}.
\end{aligned}$$

The following proposition follows from \cite[Proposition 5.1 and Proposition 5.6]{qschur-rim25}.

\begin{prop}
    Let $X$ be a real Banach space and let $c\ge1$.
    \begin{enumerate}[label=(\alph*)]
        \item The following assertions are equivalent:
        \begin{enumerate}[label=(\roman*)]
            \item $X$ has the $c$-strong Schur property;
            \item $\wca{x_n}\le c\de{x_n}$ for any bounded sequence $(x_n)$ in $X$;
            \item $\wca{x_n}\le c\wde{x_n}$ for any bounded sequence $(x_n)$ in $X$.
        \end{enumerate}
        \item If $X$ is $c$-Schur, it has also the $c$-strong Schur property.
        \item If $X$ has the $c$-strong Schur property, it is $(2c+1)$-Schur.
    \end{enumerate}
\end{prop}

We stress that the previous proposition holds only for real Banach spaces. In the complex case the situation is
more complicated and some further changes of constants are necessary (see \cite[Proposition 5.6 and Example 5.9]{qschur-rim25}). Substantial part of the proof of the previous proposition uses quantitative version of Rosenthal's $\ell_1$-theorem (see \cite{behrends} and \cite[Section 3]{qschur-rim25}). We restrict ourselves to the real spaces because Lipschitz-free spaces are over the reals.

Further, in assertion $(c)$ of the previous proposition the constant $2c+1$ cannot be replaced by $c$. This is witnessed by \cite[Example 8.3]{qschur-rim25}, where a Banach space of the form $\F(M)$ is exhibited which has the $1$-strong Schur property, is $2$-Schur but not $c$-Schur for $c<2$. However, it is not clear, whether the constant $2c+1$ is optimal or whether it can be replaced by $2c$ or $c+1$.

The previous proposition also shows that the standard terminology is a bit misleading because the $c$-strong Schur property is in fact weaker than the $c$-Schur property. But we do not attempt to change the terminology.

\section{Auxilliary results}

In this section we collect some auxilliary results which will be used to prove the main result. 
The first lemma enables us to restrict ourselves to bounded spaces. It is formulated for general metric spaces.

\begin{lemma}\label{l:bounded}
    Let $(M,d)$ be a metric space and let $c\ge 1$. Then $\F(M)$ is $c$-Schur if and only $\F(A)$ is $c$-Schur for each bounded subset $A\subset M$.
\end{lemma}

\begin{proof}
    If $A\subset M$, then $\F(A)$ is isometric to a subspace of $\F(M)$. Since the $c$-Schur property is inherited by subspaces, the `only if' part follows.

    To prove the `if part', assume that $\F(A)$ is $c$-Schur for each bounded $A\subset M$. Let $0\in M$ be a distinguished point. For $k\in\zet$ let $A_k=\{x\in M\setsep d(x,0)\le 2^k\}$. Then $A_k$ is a bounded subset of $M$, hence $\F(A_k)$ is $c$-Schur for each $k$. Thus $Y=\left(\bigoplus_{k\in\zet} \F(A_k)\right)_{\ell_1}$ is $c$-Schur as well by \cite[Theorem 7.4]{qschur-rim25}. Given $\varepsilon>0$, by \cite[Proposition 4.3] {kalton-collect04} $\F(M)$ is $(1+\varepsilon)$-isometric to a subspace of $Y$. Now we easily conclude that $\F(M)$ is $c$-Schur.
\end{proof}

The next lemma concerns only uniformly discrete spaces and enables us to restrict ourselves to integer-valued metrics.

\begin{lemma} \label{l:integer}
    Let $(M,d)$ be a uniformly discrete metric space. Given $\varepsilon>0$ there is a metric space $(M',d')$ such that $d'$ attains only integer values together with a bijection $F:M\to M'$ with $L(F)L(F^{-1})<1+\varepsilon$.
\end{lemma}

\begin{proof}
    $(M,d)$, being uniformly discrete, is $\delta$-separated for some $\delta>0$. Let $c>0$ be given. Define a new metric $d'$ on $M$ by $d'(x,y)=\lceil cd(x,y)\rceil$. It is clear that $d'$ is a metric on $M$. Moreover,
    $$ cd(x,y)\le d'(x,y)\le cd(x,y)+1\le (c+\tfrac1\delta) d(x,y).$$
    Now it is enough take $c$ so large such that $\frac{1}{c\delta}<\varepsilon$.
\end{proof}

As a consequence of these two lemmata we see that it is enough to prove the main result for bounded metric spaces with integer-valued metric. For such spaces the norm on $\F(M)$ may be expressed using integer-valued functions. This is the content of the following lemma.

\begin{lemma}
    \label{lem:extremePoints}
     Let $(M,d)$ be a metric space with integer-valued metric. Then
     \[
     \|\mu\| = \max \{\ip{\mu}{f}\colon f\in B_{\Lip_0(M)}\text{ and }f(M)\subset \zet\},\quad \mu\in \F(M).
     \]
\end{lemma}

\begin{proof} Since $\F(M)^*=\Lip_0(M)$, we have 
$$\|\mu\| = \max\{\ip{f}{\mu}\colon f\in \ext B_{\Lip_0(M)}\}, \quad \mu\in \F(M).$$
Hence, it suffices to prove that every $f\in \ext B_{\Lip_0(M)}$ attains only integer values. 
So, fix $f\in\ext B_{\Lip_0(M)}$. It follows from \cite[Theorem~1]{Fa94} that
$$\forall x\in M\;\forall\varepsilon>0\;\exists\, 0=x_0,\ldots,x_n=x\colon
\sum_{i=1}^n \Big(d(x_i,x_{i-1})-|f(x_i)-f(x_{i-1})|\Big) \leq \varepsilon.
$$
So, fix $x\in M$ and $\varepsilon>0$. Find a path $0=x_0,\ldots,x_n=x$ such that the above inequality holds. Denote 
$$s_i:=\sgn(f(x_i)-f(x_{i-1}))\mbox{ and }\varepsilon_i:=d(x_i,x_{i-1})-|f(x_i)-f(x_{i-1})|$$ for $i=1,\ldots,n$. Then for each $i$ we have $f(x_i) = s_id(x_i,x_{i-1}) - s_i\varepsilon_i + f(x_{i-1})$ and therefore, we inductively deduce that
\[
f(x) = f(x_n) = \sum_{i=1}^n s_id(x_i,x_{i-1}) - s_i\varepsilon_i,
\]
and since $\sum_{i=1}^n \varepsilon_i\leq \varepsilon$, we obtain that $f(x)$ is at the distance at most $\varepsilon$ from the integer $\sum_{i=1}^n s_id(x_i,x_{i-1})$. Since $\varepsilon>0$ was arbitrary, this proves that $d(f(x),\zet) = 0$ and therefore $f(x)\in\zet$. This completes the proof.
\end{proof}

The following easy lemma is a key tool which enables us to use in $\F(M)$ some technique from the space $\ell_1$.

\begin{lemma}\label{l:izomorfismus}
    Let $(M,d)$ be a pointed metric space which is bounded and uniformly discrete. Then $\F(M)$ is isomorphic to $\ell_1(M\setminus\{0\})$. More precisely, if $a,b>0$ are such that $a\le d(x,y)\le b$ whenever $x,y\in M$, $x\ne y$, then for each finite $F\subset M\setminus\{0\}$ and each choice $(\alpha_x)_x\in F$ of real numbers
    we have
    $$\frac a2\sum_{x\in F}\abs{\alpha_x}\le  \norm{\sum_{x\in F}\alpha_x\delta(x)}\le b \sum_{x\in F}\abs{\alpha_x}.$$
\end{lemma}

\begin{proof}
    This follows for example from \cite[Theorem 4.14]{AACD20}, which is a more general result for Lipschitz-free $p$-Banach spaces of bounded uniformly discrete quasimetric spaces. In our case (metric spaces and $p=1$) it follows easily by duality (cf. the proof of \cite[Proposition 8.2]{qschur-rim25}) as a function $f$ on $M$ satisfying $f(0)=0$ is Lipschitz if and only if it is bounded and, moreover, we clearly have
    $$\frac1b\norm{f}_\infty\le L(f)\le \frac2a\norm{f}_\infty.$$
\end{proof}

\section{Proof of Theorem~\ref{T:main}}\label{sec:proof}

In this section we provide a proof of our main result. It will be done in two steps. The first one is contained in the following lemma.

\begin{lemma}\label{L:pripaddisj}
    Let $N\in\en$. Assume $(M,d)$ is a pointed metric space such that $d$ attains only values $0,1,\dots,N$.
     Assume $(F_n)_{n=0}^\infty$ is a disjoint sequence of finite subsets of $M$. Let $(\gamma_n)_{n=0}^\infty$ be a bounded sequence in $\F(M)$ such that $\gamma_n$ is a linear combination of $\delta(x)$, $x\in F_n$. Assume that $\limsup_n \norm{\gamma_0+\gamma_n}>c>0$. Then there is a $3$-Lipschitz function $g\in\Lip_0(M)$ such that  $\limsup_n \ip{g}{\gamma_0+\gamma_n}>c$.
\end{lemma}

\begin{proof}
    Up to passing to a subsequence we may assume that $\lim_n \norm{\gamma_0+\gamma_n}=c^\prime>c$ and that
    $\norm{\gamma_0+\gamma_n}<c^\prime+1$ for each $n\in\en$.

    By Lemma~\ref{lem:extremePoints} we may find, for each $n\in\en$, a $1$-Lipschitz integer-valued function $f_n\in\Lip_0(\{0\}\cup F_0\cup F_n)$ such that $\ip{f_n}{\gamma_0+\gamma_n}=\norm{\gamma_0+\gamma_n}$. 

Since $f_n$ is $1$-Lipschitz, necessarily $\norm{f_n}_\infty\le N$. Since $F_0$ is finite, there are only finitely many integer-valued $1$-Lipschitz functions on $F_0$. Thus, up to passing to a subsequence, we may assume that $f_n|_{F_0}=f_m|_{F_0}$ for any choice $m,n\in\en$.

Note that $f_n(\{0\}\cup F_0\cup F_n)$ contains $0$ and has diameter at most $N$. Hence, it is contained in a
set of the form $\{a_n,a_n+1,\dots,a_n+N\}$ for some $a_n\in\{-N,\dots,0\}$. Up to passing to a subsequence we may assume that the sequence $(a_n)$ is constant. So, the range of each $f_n$ is contained in $\{a,\dots,a+N\}$.

Now we may glue functions $f_n$ to one function $f:\{0\}\cup\bigcup_{n=0}^\infty F_n\to\er$ by setting $f(0)=0$, $f=f_n$ on $F_0\cup F_n$. By the above it is a well-defined function. We proceed by analyzing the Lipschitz constant of $f$. By construction, $f$ is $1$-Lipschitz on $\{0\}\cup F_0\cup F_n$ for each $n\in\en$, because it coincides with $f_n$. Moreover, the range of $f$ is (contained in) $\{a,\dots,a+N\}$, so $f$ is $N$-Lipschitz. If $f$ is not $1$-Lipschitz, the Lipschitz constant is realized on pairs $(x,y)$, where $x\in F_m$, $y\in F_n$ for some $1\le m<n$. If $f$ is $3$-Lipschitz, we simply extend $f$ to a $3$-Lipschitz function $g$. So, assume $f$ is not $3$-Lipschitz. We will show how to pass to a subsequence and modify $f$ to a $3$-Lipschitz function.

To this end we set
$$A=\left\{(u,v,w)\setsep u,v\in\{a,\dots,a+N\}, w\in\{1,\dots,N\}, \frac{\abs{u-v}}{w}>3\right\}.$$

For $m,n\in \en$, $m<n$ and $\x=(u,v,w)\in A$ set
$$U_{m,n,\x}=\{x\in F_m\setsep  f_m(x)=u\ \&\ \exists y\in F_n\colon d(x,y)=w,f_n(y)=v\}. $$
 Since $F_m$ is finite, there are only finitely many possibilities for $U_{m,n,\x}$. Thus we may perform the following construction:
\begin{itemize}
    \item $P_1=\en$.
    \item Given $P_k$, set $m_k=\min P_k$.
    \item Find $U_{k,\x}\subset F_{m_k}$ for each $\x\in A$ and an infinite set $P_{k+1}\subset P_k\setminus\{m_k\}$ such that $U_{m_k,n,\x}=U_{k,\x}$ for each $n\in P_{k+1}$ and each $\x\in A$.
\end{itemize}

So, passing to a subsequence we may assume that $U_{m,n,\x}=U_{m,\x}$ whenever $n>m$ and $\x\in A$. 
Up to passing to a further subsequence we may assume that, given $\x\in A$ either all sets $U_{m,\x}$ are empty
or all sets $U_{m,\x}$ are nonempty. Let $\widetilde{A}$ be the subset of $A$ formed by those $\x$ for which sets $U_{m,\x}$ are nonempty.

If $\widetilde{A}=\emptyset$, then $f$ is $3$-Lipschitz, hence we let $g$ to be a $3$-Lipschitz extension of $f$.

Next assume that $\widetilde{A}\ne\emptyset$.
For $m,n\in\en$, $m<n$ and $\x=(u,v,w)\in \widetilde{A}$ set
$$V_{m,n,\x}=\{y\in F_n\setsep \exists x\in U_{m,\x}\colon d(x,y)=w, f_n(y)=v\}.$$
By construction this set is always nonempty.
We claim that
\begin{equation}\label{eq:disjointness}
 V_{m_1,n,\x}\cap V_{m_2,n,\x}= \emptyset \mbox{ whenever } m_1<m_2<n,\x\in\widetilde{A}.  
\end{equation}
Indeed, assume $m_1<m_2<n$, $\x=(u,v,w)\in\widetilde{A}$ and $z\in V_{m_1,n,\x}\cap V_{m_2,n,\x}$.
Then there are $x_1\in U_{m_1,\x}$ and $x_2\in U_{m_2,\x}$ such that $d(x_1,z)=d(x_2,z)=w$. Moreover, fix
$y\in F_{m_2}$ such that $d(x_1,y)=w$ and $f_{m_2}(y)=v$. (Such $y$ does exist as $x_1\in U_{m_1,\x}$.) Then
$$d(x_2,y)\le d(x_2,z)+d(z,x_1)+d(x_1,y)=3w.$$
Further, $f_{m_2}(x_2)=u$, $f_{m_2}(y)=v$ and hence
$$\abs{u-v}=\abs{f_{m_2}(x_2)-f_{m_2}(y)}\le d(x_2,y)\le 3w,$$
so 
$$\frac{\abs{u-v}}{w}\le3,$$
a contradiction with $(u,v,w)\in\widetilde{A}\subset A$.

Let $\gamma_n=\sum_{x\in F_n}\alpha_x\delta(x)$. 
By Lemma~\ref{l:izomorfismus} we know that 
$$\forall n\in\en:\quad \sum_{x\in F_n}\abs{\alpha_x}\le 2(c^\prime+1).$$
For $m<n$ and $\x\in\widetilde{A}$ set
$$c(m,n,\x)=\sum_{x\in V_{m,n,\x}} \abs{\alpha_x}.$$
Note that $c(m,n,\x)\le 2(c^\prime+1)$ for any $m,n,\x$, so
 up to passing to a subsequence we may assume that $c(m,n,\x)\to c(m,\x)$ for each $m$ and each $\x\in\widetilde{A}$.

Further, for any $\x\in\widetilde{A}$, using \eqref{eq:disjointness} we obtain that the sequence 
$$(\sum_{m<n} c(m,n,\x))_n$$
is also bounded by $ 2(c^\prime+1)$. We deduce that for every $k\in\en$ we have
\[
\sum_{m=1}^k c(m,\x) = \sum_{m=1}^k \lim_{n\to\infty} c(m,n,\x) = \lim_{n\to\infty} \sum_{m=1}^{k} c(m,n,\x)\leq  2(c^\prime+1)\]
so $\sum_{m=1}^\infty c(m,\x)\leq  2(c^\prime+1)< \infty$ every $\x\in\widetilde{A}$.

Fix $\varepsilon>0$ small enough (to be specified later). Let $m_0$ be such that 
$$\sum_{m\ge m_0}c(m,\x)<\varepsilon \mbox{ for each }\x\in\widetilde{A}.$$
We define an increasing sequence $(k_j)$ of natural numbers as follows.
Set $k_1=m_0$. Given $k_j$, find $k_{j+1}>k_{j}$ such that $c(k_i,n,\x)<c(k_i,\x)+\frac{\varepsilon}{2^{j+1}}$ for $\x\in\widetilde{A}$, $i\le j$ and $n\ge k_{j+1}$.

Set
$$H=\{0\}\cup F_0 \cup \bigcup_{j=1}^\infty F_{k_j} \setminus \bigcup_{i<j,\x\in\widetilde{A}}V_{k_i,k_j,\x}.$$
Then $f|_H$ is $3$-Lipschitz. So, there is a $3$-Lipschitz function $g$ on $M$ extending $f|_H$. Then
$$\begin{aligned}
    \ip{g}{\gamma_0+\gamma_{k_l}}&=\norm{\gamma_0+\gamma_{k_l}}+ \ip{g-f}{\gamma_0+\gamma_{k_l}}
    \\&\ge \norm{\gamma_0+\gamma_{k_l}} - \sum_{\x\in \widetilde{A}}\sum_{j<l}\sum_{x\in V(k_j,k_l,\x)}\abs{\alpha_x}\abs{f(x)-g(x)}
     \\&\ge \norm{\gamma_0+\gamma_{k_l}} - 4N\sum_{\x\in \widetilde{A}}\sum_{j<l}\sum_{x\in V(k_j,k_l,\x)}\abs{\alpha_x}
   \\& =\norm{\gamma_0+\gamma_{k_l}} -4N \sum_{\x\in \widetilde{A}} \sum_{j<l} c(k_j,k_l,\x) 
   \\& > \norm{\gamma_0+\gamma_{k_l}} - 4N \sum_{\x\in \widetilde{A}}\sum_{j<l} (c(k_j,\x)+\tfrac{\varepsilon}{2^l})\\& >  \norm{\gamma_0+\gamma_{k_l}} - 4N \sum_{\x\in \widetilde{A}}\varepsilon(1+\tfrac{l}{2^l})
   \\&\ge  \norm{\gamma_0+\gamma_{k_l}} - 4N(N+1)^3\varepsilon(1+\tfrac{l}{2^l})
\end{aligned}$$
Here, the first equality follows from the choice of $f$. The next inequality follows from the fact that $f=g$ on $H$ and the triangle inequality. In the following one we use that $\abs{f}\le N$ and $\abs{g}\le 3N$ (as $g$ is $3$-Lipschitz and $g(0)=0$). The next equality follows from the definition of $c(m,n,\x)$. The following inequality follows from the construction of $(k_j)$ and the next one follows from the choice of $k_1=m_0$. The last inequality follows from the fact that cardinality of $\widetilde{A}$ is at most $(N+1)^3$.

Hence $$\limsup\ip{g}{\gamma_0+\gamma_n}\ge\liminf \ip{g}{\gamma_0+\gamma_{k_l}}\ge c^\prime-4N(N+1)^3\varepsilon>c$$
if $\varepsilon$ is small enough. This completes the proof.
\end{proof}

The second step is contained in the following lemma, which enables us to reduce the situation to the setting of the preceding lemma.

\begin{lemma}\label{L:redukce}
    Let $N\in\en$. Assume $(M,d)$ is a pointed metric space such that $d$ attains only values $0,1,\dots,N$.
     Assume $(\mu_n)$ is a bounded sequence in $\F(M)$ such that $\limsup_n \norm{\mu_n}>c>0$. Then there is a $3$-Lipschitz function $g\in\Lip_0(M)$ such that  $\limsup_n \ip{g}{\mu_n}>c$.
\end{lemma}

\begin{proof} Let $T:\F(M)\to \ell^1(M\setminus\{0\})$ be the isomorphism provided by Lemma~\ref{l:izomorfismus}. Note that
$T(\delta(x))=e_x$ for $x\in M\setminus\{0\}$. 
 Using the identification via $T$, we transfer from $\ell_1(M\setminus\{0\})$ to $\F(M)$ the notions of a support, $w^*$-convergence and restriction to a subset of $M$. We recall that $w^*$-topology on bounded subsets of $\ell_1$ coincides with the topology of pointwise convergence.

Let $(\mu_n)$ and $c$ be as in the statement.  Up to passing to a subsequence we may assume that $\lim\norm{\mu_n}=c^\prime\in (c,\infty)$. Up to passing to a further subsequence we may assume that $\mu_n\stackrel{w^*}{\longrightarrow} \mu$ for some $\mu\in \F(M)$ with $\|\mu\|\leq c^\prime$. Fix $\varepsilon>0$ such that $c^\prime-9\varepsilon>c$. Pick a finite $F_0\subset M\setminus\{0\}$ with $\|\mu|_{M\setminus F_0}\| < \varepsilon$. Since $\mu_n\to\mu$ pointwise and $F_0$ is finite, up to omitting a finite number of elements from the sequence, we may assume that $\|(\mu_n-\mu)|_{F_0}\|<\varepsilon$ for $n\in\en$.

Now, we shall show that passing to a subsequence we may assume there are pairwise disjoint finite sets $F_n\subset M\setminus F_0$ such that 
\begin{equation}\label{eq:glidingBumb}
\|\mu_n|_{M\setminus(F_0\cup F_n)}\| < \varepsilon,\quad n\in\en.
\end{equation}
Indeed, we start by putting $n_1:=1$ and picking finite $F_1\subset M\setminus F_0$ satisfying $\|\mu_1|_{M\setminus F}\| < \varepsilon$ whenever $F\supset F_0\cup F_1$. For the inductive step, assume $n_j\in \en$ and $F_j\subset M$ for $j\leq k$ were defined. Since $\|\mu|_{M\setminus F_0}\| < \varepsilon$ and $\mu_n\to \mu$ pointwise, we pick $n_{k+1}>n_k$ such that for every $n\geq n_{k+1}$ we have $\|\mu_n|_{\bigcup_{i=1}^kF_i}\| < \tfrac{\varepsilon}{2}$. Next, we pick $F_{k+1}\subset M\setminus (F_0\cup\ldots\cup F_k)$ such that for every $F\supset \bigcup_{i=1}^{k+1} F_i$ we have $\|\mu_{n_{k+1}}|_{M\setminus F}\| < \tfrac{\varepsilon}{2}$. Then $\|\mu_{n_{k+1}}|_{M\setminus (F_0\cup F_{k+1})}\| < \varepsilon$. This finishes the inductive construction.

Denote $\gamma_0:=\mu|_{F_0}$ and $\gamma_n:=\mu_n|_{F_n}$ for $n\in\en$, by the above those are points in $\F(M)$ with finite disjoint supports. Further, we have
\[
\|\gamma_0+\gamma_n\|\geq \|\gamma_0 + \mu_n|_{M\setminus F_0}\| - \varepsilon > \|\mu_n|_{F_0} + \mu_n|_{M\setminus F_0}\| - 2\varepsilon = \|\mu_n\| - 2\varepsilon.
\]
Consequently, $\limsup_n \norm{\gamma_0+\gamma_n}\ge c^\prime-2\varepsilon$. By Lemma~\ref{L:pripaddisj} there is a $3$-Lipschitz function $g\in\Lip_0(M)$ with $\limsup\ip{g}{\gamma_0+\gamma_n}>c^\prime-3\varepsilon$.
For every $n\in\en$ we have
\[\begin{split}
\ip{g}{\mu_n} & = \ip{g}{\mu|_{F_0}} + \ip{g}{\mu_n|_{F_n}} +\ip{g}{(\mu_n-\mu)|_{F_0}}+ \ip{g}{\mu_n|_{M\setminus (F_0\cup F_n)}}\\
& \geq \ip{g}{\gamma_0+\gamma_n}-3\norm{(\mu_n-\mu)|_{F_0}}-3\norm{\mu_n|_{M\setminus (F_0\cup F_n)}} \\
& >\ip{g}{\gamma_0+\gamma_n}-6\varepsilon.
\end{split}\]
Hence
$$\limsup\ip{g}{\mu_n}\ge \limsup\ip{g}{\gamma_0+\gamma_n}-6\varepsilon>c^\prime-9\varepsilon>c.$$
This completes the proof.
\end{proof}

Now we are ready to complete the proof of the main theorem.

\begin{proof}[Proof of Theorem~\ref{T:main}.]
Assume that $M$ is a uniformly discrete metric space. We want to prove that $\F(M)$ is $3$-Schur. By Lemma~\ref{l:bounded} we may assume that $M$ is bounded. By Lemma~\ref{l:integer} we may additionally assume
that the metric on $M$ attains only integer values. In this case $M$ satisfies the assumptions of Lemma~\ref{L:redukce}, which shows that condition from Lemma~\ref{l:c-schur-char} is satisfied with $c=3$.
Hence $\F(M)$ is $3$-Schur.

Optimality of constant $3$ follows from \cite[Example 8.5]{qschur-rim25}
\end{proof}

\section{Proofs of Theorem~\ref{t:0hyp} and its corollary}\label{sec:0hyp}

Theorem~\ref{t:0hyp} follows essentially by a combination of known results, it follows immediately from a formally stronger Theorem~\ref{thm:0HypAndp1u} below.

Before proceeding to the proofs, let us recall that a metric space is $0$-hyperbolic if and only if it embeds into a an $\mathbb{R}$-tree; see \cite[Theorems 3.38]{Evans08}. Rather than recalling the full definition of an $\er$-tree, which requires additional notions, we note the following convenient characterization: a metric space $(T,d)$ is an $\er$-tree if and only if it is $0$-hyperbolic and connected, see \cite[Theorem 3.40]{Evans08}.

Assume now that $(T,d)$ is an $\er$-tree. Then for any $x,y\in T$ there is a unique isometry $\phi_{x,y}:[0,d(x,y)]\to T$ with $\phi_{x,y}(0)=x$ and $\phi_{x,y}(y)=d(x,y)$, see \cite[Lemma 3.20]{Evans08}. Finally, given $A\subset T$, we say $A$ has length measure zero if $\lambda(\phi_{x,y}^{-1}(A))=0$ for every $x,y\in T$. We refer the interested reader to \cite{Evans08} for more details concerning $\mathbb{R}$-trees and to \cite{God10} for more details concerning the length measure on $\mathbb{R}$-trees and results related to Lipschitz-free spaces in this context.

\begin{thm}\label{thm:0HypAndp1u}
    Let $M$ be a separable complete metric space. Then the following conditions are equivalent.
    \begin{enumerate}[label=(\alph*)]
        \item\label{it:almostInEll1} $\F(M)$ is $(1+\varepsilon)$-isomorphic to a subspace of $\ell_1$ for every $\varepsilon>0$.
        \item\label{it:metricAlmostInEll1} $M$ is $0$-hyperbolic and purely 1-unrectifiable.
    \end{enumerate}
Consequently, if $M$ is complete $0$-hyperbolic and purely 1-unrectifiable (and not neccessarily separable), then $\F(M)$ is $1$-Schur.
\end{thm}
\begin{proof}The proof follow from two results proved elsewhere. First, by \cite[Theorem 1.2]{APP21}, condition \ref{it:almostInEll1} holds if and only if $M$ embeds into an $\er$-tree as a set of length measure zero. Second, by \cite[Theorem C]{p1u} $\F(M)$ is Schur if and only if $M$ is purely $1$-unrectifiable. Knowing those two results, the proof is almost immediate.

Indeed, if \ref{it:almostInEll1} holds then $\F(M)$ is Schur, because it isomorphically embeds into $\ell_1$. Hence, $M$ is purely $1$-unrectifiable. Moreover, $M$ is $0$-hyperbolic because it embeds into an $\er$-tree. Hence, \ref{it:metricAlmostInEll1} holds.

Conversely, if \ref{it:metricAlmostInEll1} holds then, since $M$ is $0$-hyperbolic, it embeds into an $\er$-tree and if $(T,d)$ is an $\er$-tree and $M\subset T$ is purely 1-unrectifiable, then it has lenght measure zero as otherwise there would be $x,y\in T$ with $\lambda(\phi_{x,y}^{-1}(M))>0$, so there would exist a compact set $K\subset \phi_{x,y}^{-1}(M)\subset [0,d(x,y)]$ of positive measure and $\phi_{x,y}(K)\subset M$ would be an isometric copy of a compact subset of the real line with positive measure, contradicting that $M$ is purely 1-unrectifiable. Hence, $M$ embeds into an $\er$-tree as a set of length measure zero, so \ref{it:almostInEll1} holds.

For the ``Consequently'' part, assume $M$ is complete $0$-hyperbolic and purely 1-unrectifiable. Since $1$-Schur property can be obviously tested on separable subspaces and since for any separable subspace $Y\subset \F(M)$ we can find $N\subset M$ complete and separable such that $Y\subset \F(N)$, we may without loss of generality assume $M$ is separable. Then, by the already proven part we obtain that $\F(M)$ is $(1+\varepsilon)$-isomorphic to a subspace of $\ell_1$ for every $\varepsilon>0$, which implies that $\F(M)$ is $1$-Schur.
\end{proof}

A special case when the above applies is when $M$ is an ultrametric space. We are convinced it is well-known that ultrametric spaces are purely $1$-unrectifiable, but we were not able to find a reference with a suitable proof. For the convenience of the reader, we therefore include a proof below. We also refer to Remark~\ref{rem:ultraMPJinak} for an alternative approach to deriving Lemma~\ref{L:um-p1u} from existing results; however, this approach appears somewhat cumbersome.

\begin{lemma}\label{L:um-p1u}
    Any ultrametric space is purely $1$-unrectifiable.
\end{lemma}

 \begin{proof}
 {\tt Step 1.} Assume that $K\subset\er$ is a compact set of positive measure. Given $\varepsilon>0$, there are real numbers $a<b$ such that $\lambda(K\cap [a,b])>(1-\varepsilon)(b-a)$.

\smallskip

This is an immediate consequence of Lebesgue density theorem. But in fact, this statement is elementary, so we provide an easy proof.

The assertion is obvious if $K$ contains an interval. So, assume it does not contain an interval, i.e., it is totally
disconnected. Set $\alpha=\min K$, $\beta=\max K$ and $m=\lambda(K)$. Then $\alpha<\beta$, in fact $\beta-\alpha>m$. 
$[\alpha,\beta]\setminus K$ is a nonempty open set, so it is a disjoint union of open intervals. This family of intervals must be countable (as we are in $\er$) and infinite (otherwise $K$ contains an interval). So, say
$$[\alpha,\beta]\setminus K=\bigcup_{n=1}^\infty (\alpha_n,\beta_n).$$
For $n\in\en$ define
$$\varepsilon_n=\lambda\Bigg([\alpha,\beta]\setminus \Big(K\cup \bigcup_{k=1}^n(\alpha_n,\beta_n)\Big)\Bigg)=\beta-\alpha-m-\sum_{k=1}^n(\beta_n-\alpha_n).$$
clearly $\varepsilon_n\to0$.
Further, set
$$A_n=[\alpha,\beta]\setminus \bigcup_{k=1}^n(\alpha_n,\beta_n).$$
Then $K\subset A_n$ and
$$\frac{\lambda(K)}{\lambda(A_n)}=\frac{m}{m+\varepsilon_n}.$$
Moreover, $A_n$ is a finite union of singletons and closed intervals. So, there is one of these intervals, say $[a,b]$ such that
$$\lambda(K\cap[a,b])\ge\frac{m}{m+\varepsilon_n}(b-a)>0.$$
Since $\frac{m}{m+\varepsilon_n}\nearrow 1$, the assertion follows.

\smallskip

 {\tt Step 2.} Assume that $K\subset\er$ is a compact set of positive measure. Let $d$ be a metric on $K$ such that $(K,d)$ is ultrametric and $d(x,y)\le \abs{x-y}$ for $x,y\in K$. Given $\varepsilon>0$ there are $x,y\in K$ such that $d(x,y)<\varepsilon\abs{x-y}$.

 \smallskip

 Fix $n\in\en$, $n\ge 3$. By Step 1 there are real numbers $a<b$ such that
 $\lambda(K\cap[a,b])>(1-\frac1n)(b-a)$. Let $a=t_0<t_1<\dots<t_n=b$ be the equidistant partition. By the choice of $a,b$ for each $j\in\{1,\dots,n\}$ there is some $x_j\in K\cap (t_{j-1},t_j)$. Then $d(x_j,x_{j+1})\le \abs{x_{j+1}-x_j}<\frac2n(b-a)$ for $j=1,\dots,n-1$. The ultrametric property yields $d(x_1,x_n)<\frac2n(b-a)$.
 On the other hand, $\abs{x_1-x_n}>(1-\frac2n)(b-a)$.
Thus, $$\frac{d(x_1,x_n)}{\abs{x_1-x_n}}<\frac{\frac2n}{1-\frac2n}=\frac{2}{n-2}.$$
Since $\frac{2}{n-2}\to 0$, the assertion follows.

\smallskip

{\tt Step 3.} Completion of the proof.

\smallskip

Assume $(M,d)$ is ultrametric which is not purely $1$-unrectifiable. Then there is a compact $K\subset \er$ of positive measure and a bi-Lipschitz embedding $f:K\to M$. 
For $x,y\in K$ define $\rho(x,y)=d(f(x),f(y))$. Then $(K,\rho)$ is ultrametric and $\rho$ is Lipschitz-equivalent to the euclidean metric. This contradicts Step 2.
 
 \end{proof}

 \begin{remark}\label{rem:ultraMPJinak}An alternative approach to deriving Lemma~\ref{L:um-p1u} from existing results is the following. Pick an ultrametric space $M$. Since completion of an ultrametric space is ultrametric, we may without loss of generality assume $M$ is complete. Moreover, obviously $M$ is purely $1$-unrectifiable if and only if each of its compact subsets is, so we may assume $M$ is compact. At this point, there are at least two ways to conclude. First, by \cite[Theorem 2]{CD16} the space $\F(M)$ is isomorphic to $\ell_1$, so it is Schur and $M$ is purely $1$-unrectifiable by \cite[Theorem D]{p1u}. Alternatively, one may apply \cite[Theorem 3.7]{Dal15} which asserts that $\F(M)$ is isometric to a dual space, and then invoke \cite[Theorem B]{p1u}.
 \end{remark}

\begin{proof}[Proof of Corollary~\ref{cor:ultrametric}.]
Let $M$ be ultrametric. Its completion is easily seen to be ultrametric as well, so without loss of generality $M$ is complete. It is clear that $M$ is $0$-hyperbolic. Further, by 
Lemma~\ref{L:um-p1u} $M$ is purely $1$-unrectifiable, so we may conclude by Theorem~\ref{t:0hyp}.    
\end{proof}
 
\def\cprime{$'$} \def\cprime{$'$}

\end{document}